\documentclass[reqno]{amsart}
\usepackage[margin=3.5cm]{geometry}
\usepackage{xcolor}
\usepackage[T1]{fontenc}
\usepackage{amsfonts, amsthm, amssymb,verbatim,amscd,amsmath, blindtext}

\usepackage{multirow}
\usepackage{graphicx}
\usepackage{enumitem,yhmath}
\usepackage{amssymb}
\usepackage{ytableau}
\usepackage{tikz}
\usepackage{tikz-cd}
\usepackage{float, pifont}
\usepackage{mathtools}
\usepackage[pagebackref,colorlinks=true,citecolor=blue,linkcolor=black]{hyperref}

\def\Hom{\operatorname{Hom}}

\DeclareMathOperator{\characteristic}{char}
\DeclareMathOperator{\perm}{perm}
\DeclareMathOperator{\initial}{in}
\DeclareMathOperator{\depth}{depth}
\DeclareMathOperator{\on}{on}
\DeclareMathOperator{\off}{off}

\newtheorem{theorem}{Theorem}[section]
\newtheorem{lemma}[theorem]{Lemma}
\newtheorem{proposition}[theorem]{Proposition}
\newtheorem{corollary}[theorem]{Corollary}

\theoremstyle{definition}

\theoremstyle{remark}

\numberwithin{equation}{section}

\newcommand{\reqnomode}{\tagsleft@false\let\veqno\@@eqno}

\begin{document}

\title{Permanental ideals of symmetric matrices}

\author{Trung Chau}
\address{Chennai Mathematical Institute, Siruseri, Tamil Nadu, India}
\email{chauchitrung1996@gmail.com}



\begin{abstract}
    In this article, we study the ideal generated by $2\times 2$ permanents of a symmetric matrix. We denote this ideal by $P_2(X)$ where $X$ is a symmetric matrix. We compute a Gr\"obner basis, dimension, depth, minimal primes, and a primary decomposition of $P_2(X)$. It can be seen that the answer is reliant on whether the characteristic of the base field is two, and thus these ideals constitute a class of ideals whose algebraic properties depend on characteristics of the base field.
\end{abstract}
\maketitle

\section{Introduction}

Let $X=(x_{ij})$ be an $n\times n$ matrix. The \emph{permanent} of $X$ is the polynomial
\[
\perm(X)=\sum_{\sigma} x_{1,\sigma(1)}\cdots x_{n,\sigma(n)},
\]
where the summation ranges over all the permutations of the set $\{1,\dots, n\}$. Permanents, loosely speaking, are determinants with sums instead of alternating sums. Despite the similarity in structures, their behaviors and properties are wildly different. A fundamental difference is that determinants are largely unchanged after an elementary row operation, while permanents are not. This leads to algorithms that compute determinants in polynomial time (e.g., one that uses the LU decomposition by Banachiewicz in 1938). On the other hand, the computation of the permanents of $(0,1)$-matrices is shown to be NP-hard \cite{VALIANT1979189}. The existence of an algorithm that computes permanents in polynomial time thus would imply that P=NP, the heart of complexity theory. Permanents, introduced by Cauchy and Binet in early 1800s, have since found applications in geometry, computer science, combinatorics, and probability theory. We refer to \cite{Aaronson:14, Fano,EGORYCHEV1981299} for some such applications and \cite{Minc} for a survey on the subject.

On the algebraic side, determinantal ideals of a matrix $X$ with indeterminates as entries, i.e., ideals generated by determinants of $t\times t$  square submatrices of $X$ in the polynomial ring $\Bbbk[X]$, where $t$ is a fixed integer and $\Bbbk$ is a field, have been intensively studied (see, e.g., \cite{BCRV, BV}). The cases when $X$ is generic \cite{Eagon1971InvariantTA}, symmetric \cite{Kutz74}, and Hankel \cite{CONCA2018111} are of special interests, and many algebraic properties of determinantal ideals have been obtained in these cases. On the other hand, \emph{permanental ideals}, i.e., ideals generated by permanents of $t\times t$  square submatrices of $X$ in the polynomial ring $\Bbbk[X]$, have not enjoyed the same level of popularity, partly due to their complexity. We denote this permanental ideal by $P_t(X)$. Some studied cases include $P_2(X)$ when $X$ is generic \cite{LS00, DK21} or Hankel \cite{GGS07}, $P_3(X)$ when $X$ is generic \cite{Kirkup} or Hankel \cite{GG06}, among other cases \cite{BCMV25,ELSW18,HMSW15}. It is important to note that when $\characteristic \Bbbk =2$, the ideal $P_t(X)$ is the same as the corresponding determinantal ideal. The same cannot be said when $\characteristic \Bbbk \neq 2$. Thus $P_t(X)$ represents a class of ideals whose properties depend on charactersitics of the base~field.

In this paper, we study the permanental ideal $P_2(X)$ in the polynomial ring $\Bbbk[X]$ where $X=(x_{ij})$ is an $n\times n$ symmetric matrix. Throughout this paper, we assume that $\characteristic \Bbbk \neq 2$. The properties of $P_2(X)$ turn out to be very different from their determinantal counterpart, and also from the case when $X$ is generic or Hankel. For example, we show that the generators of $P_2(X)$ do not form a Gr\"obner basis, unlike its determinantal counterpart \cite{CONCA1994406}. As another example, the permanental ideal $P_2(Y)$, where $Y$ is either generic or Hankel, has exactly one embedded associated prime, being the homogeneous maximal ideal \cite{GGS07,LS00}, while as we shall see in this paper, $P_2(X)$ has exactly $n+1$ embedded associated primes.

The paper is structured as follows. Section~\ref{sec:2} includes some basic results on monomials $P_2(X)$ contain. Section~\ref{sec:3} gives a Gr\"obner basis for $P_2(X)$ with respect to a lexicographical diagonal monomial ordering in Theorem~\ref{thm:Grobner-symmetric}. In section~\ref{sec:4}, the goal is to find all minimal primes of $P_2(X)$ (Theorem~\ref{thm:symmetric-minimal-primes})  and use them to compute $\sqrt{P_2(X)}$ (Proposition~\ref{prop:symmetric-generators-radical}). Finally, Section~\ref{sec:5} proves our main theorem, an irredundant primary decomposition of $P_2(X)$ in Theorem~\ref{thm:primary-decomposition}.

\section{Monomials in $P_2(X)$}\label{sec:2}

Ideals generated by $2\times 2$ subpermanents of a matrix with indeterminates as entries contain many monomials, a remarkable difference from determinantal ideals. We recall two results that illustrate this from \cite{LS00}. It is worth noting that even though generic matrices were the main focus in \cite{LS00}, the cited lemmas below work for any matrices. We provide a proof for completeness.

\begin{lemma}[{\cite[Lemma 2.1]{LS00}}]\label{lem:three-element-in-permanent}
	The ideal $P_2(X)$ contains all products of three entries of $X$, taken from three distinct columns and two distinct rows, or from two distinct columns and three distinct rows.
\end{lemma}

\begin{proof}
    Without loss of generality, we can assume that $X$ is a $2\times 3$ matrix
    \[
    X=\begin{pmatrix}
        x & y & z\\
        w & u & v
    \end{pmatrix}
    \]
    and it suffices to show that $xyv\in P_2(X)$. Indeed, we have
    \[
    2xyv= x(yv+zu) - z(xu+yw)  + y(zw+xv) \in P_2(X), 
    \]
    and thus $xyv\in P_2(X)$, as claimed.
\end{proof}

\begin{lemma}[{\cite[Lemma 2.2]{LS00}}]\label{lem:three-element-in-permanent-2}
	The ideal $P_2(X)$ contains all products of the form $x_{i_1j_1}^2x_{i_2j_2}x_{i_3j_3}$ with distinct $i_1,i_2,i_3$ and distinct $j_1,j_2,j_3$.
\end{lemma}

\begin{proof}
    Without loss of generality, we can assume that $X$ is a $3\times 3$ matrix
    \[
    X=\begin{pmatrix}
        x_1 & x_2 & x_3\\
        y_1 & y_2 & y_3\\
        z_1 & z_2 & z_3
    \end{pmatrix}
    \]
    and it suffices to show that $x_1^2y_2z_3\in P_2(X)$. Indeed,  by Lemma \ref{lem:three-element-in-permanent}, we have
    \[
    x_1^2y_2z_3= x_1z_3(x_1y_2+x_2y_1) - y_1(x_1x_2yz_3) \in P_2(X),
    \]
    as claimed.
\end{proof}

In the case when $X$ is symmetric, some more products of special forms are in $P_2(X)$.

\begin{lemma}\label{lem:three-element-symmetric}
    Let $X$ be a symmetric matrix. Then the ideal $P_2(X)$ contains all products of the form $x_{ij}^2x_{ik}$, $x_{ij}^2x_{kk}^2$, or $x_{ij}^3x_{kk}$ with distinct $i,j,k$.
\end{lemma}
\begin{proof}
    It suffices to assume that $i=1,j=2,$ and $k=3$. We look at the positions of the variables in the submatrix of $X$ formed by its first three rows and three columns:
    \[
    \begin{pmatrix}
        * & x_{ij} & x_{ik}\\
        x_{ij} & * & *\\
        *&*&x_{kk}
    \end{pmatrix}.
    \]
    Thus $x_{ij}^2x_{ik}\in P_2(X)$ by Lemma~\ref{lem:three-element-in-permanent}, and $x_{ij}^2x_{kk}^2,x_{ij}^3x_{kk}\in P_x(X)$ by Lemma~\ref{lem:three-element-in-permanent-2}, as desired.
\end{proof}

\begin{lemma}\label{lem:quadratic-generator}
    Let $X$ be a symmetric matrix. Then then ideal $P_2(X)$ contains all products of the form $x_{ij}x_{kl}$  with distinct $i,j,k,l$.
\end{lemma}
\begin{proof}
    We have
	\begin{align*}
		x_{ij}x_{kl}+x_{il}x_{jk}=\perm\begin{pmatrix}
			x_{ij}&x_{il}\\
			x_{kj}&x_{kl}
		\end{pmatrix}\in P_2(X),\\
		x_{ik}x_{jl}+x_{il}x_{jk}=\perm\begin{pmatrix}
			x_{ik}&x_{il}\\
			x_{jk}&x_{jl}
		\end{pmatrix}\in P_2(X),\\
		\text{and } x_{ij}x_{kl}+x_{ik}x_{jl}=\perm\begin{pmatrix}
			x_{ij}&x_{ik}\\
			x_{lj}&x_{lk}
		\end{pmatrix}\in P_2(X).
	\end{align*}
    Hence $2x_{ij}x_{kl}\in P_2(X)$. Since $2$ is invertible, we have $x_{ij}x_{kl}\in P_2(X)$.
\end{proof}

\section{A Gr\"obner basis of $P_2(X)$}\label{sec:3}

In this section, we compute a reduced Gr{\"o}bner basis for $P_2(X)$ with respect to any lexicographic (lex) diagonal ordering of monomials. Recall that a monomial ordering on the variables is called \emph{diagonal} if for any square submatrix of $X$, the leading term of the permanent of this submatrix is exactly the product of the entries on its diagonal. Such an ordering is known to exist for symmetric matrices, such as one defined by
\[
x_{ij}>x_{kl} \text{ if $i<k$ or $i=k$ and $j<l$}, 
\]
where $i\leq j$ and $k\leq l$.

\begin{theorem}\label{thm:Grobner-symmetric}
	Let $X$ be an $n\times n$ symmetric matrix with indeterminates as entries. The following collection $G$ of polynomials is a reduced Gr{\"o}bner basis for $P_2(X)$ with respect to any diagonal ordering:
	\begin{enumerate}
		\item[(1a)]  The subpermanents $x_{ii}x_{jj}+x_{ij}^2$,\  $i<j$;
		\item[(1b)]  the subpermanents $x_{ii}x_{jk}+x_{ij}x_{ik}$,\  $j<k$,\  $i\neq j,k$;
		\item[(1c)]  $x_{ij}x_{kl}$,\  $i,j,k,l$ are distinct, $i<j$, $k<l$;
		\item[(2a)]  $x_{il}x_{jl}x_{kl}$,\  $i<j<l$,\  either $k=i$, $k=j$, or $j< k< l$;
		\item[(2b)]  $x_{il}x_{jl}x_{kk}$,\  $i<j<k<l$;
		\item[(2c)]  $x_{ij}x_{ik}x_{jj}$,\  $i<j<k$;
		\item[(3a)]  $x_{ij}x_{ik}x_{il}$,\  $i<k<l$,\  either $j=l$, $j=k$, or $i< j< k$;
		\item[(3b)]  $x_{ik}x_{il}x_{jj}$,\  $i<j<k<l$;
		\item[(3c)]  $x_{ik}x_{jk}x_{jj}$,\  $i<j<k$;
		\item[(6a)]  $x_{ik}^3x_{jj}$,\  $i<j<k$;
		\item[(6b)]  $x_{ik}^2x_{jj}^2$,\ $i<j<k$.
	\end{enumerate}
\end{theorem}

This Gr{\"o}bner basis is similar to that of $P_2(Y)$, where $Y$ is a generic matrix, found in \cite[Theorem 3.2]{LS00}. We use corresponding notations to emphasize this similarity. For example, polynomials of types $(1a),(1b),(1c)$ are minimal generators of $P_2(X)$, and they correspond to polynomials of types $(1)$ in \cite[Theorem 3.2]{LS00} which are minimal generators of $P_2(Y)$. Thus, to simplify the notations, we will call polynomials of types $(1a), (1b), (1c)$ to be of type $(1)$, and we apply similar definitions to other types. We also remark that we use $(6a)$ instead of $(4a)$ since these monomials are similar to those of type $(6)$ in \cite[Theorem 3.2]{LS00}.

Pictorially, the monomials of types $(2)$, $(3)$ and $(6)$ are products of the following entries (each entry may repeat itself) of suitably sized submatrices:
\begin{center}
	\begin{tabular}{cccccc}
		$\begin{pmatrix}
			&& \circ \\
			* &\circ &
		\end{pmatrix}$ &\ \ \ \ \ \ \ & $\begin{pmatrix}
			&\circ & * \\
			\circ & &
		\end{pmatrix}$ &\ \ \ \ \ \ \ \ \ & $\begin{pmatrix}
			& \circ \\
			\bullet &
		\end{pmatrix}$ \\
		type $(2)$              &\ \ \ \ \ \ \ \ \ &  type $(3)$ &\ \ \ \  \ \ \ \ \ & type $(6)$      
	\end{tabular}
\end{center}
Here $\circ$ marks entries that can be anywhere, $*$ entries that are off-diagonal, and $\bullet$ entries that are on the diagonal. Some monomials, at a first glance, may not look like what the pictures show. For example, an element of type $(2a)$ is $x_{il}x_{jl}x_{kl}$, and a product of three entries in the same column; but it equals $x_{li}x_{lj}x_{kl}$, the product of entries in the picture for type-(2) monomials.

There are two reasons why this Gr{\"o}bner basis is, in a sense, more complicated than its generic counterpart: The same permanent may appear in different ways in a symmetric matrix, and the appearance of quadratic monomial minimal generators (Proposition \ref{lem:quadratic-generator}) forces the polynomials in a Gr{\"o}bner basis to be more selective.

Finally, we remark that it is easy to compute the cardinality of this Gr{\"o}bner basis, which~is
\[
\binom{n}{2}+11\binom{n}{3}+7\binom{n}{4}.
\]

\begin{proof}[Proof of Theorem \ref{thm:Grobner-symmetric}]
	It is clear that each polynomial of $G$ is in $P_2(X)$ by Lemmas \ref{lem:three-element-in-permanent}, \ref{lem:three-element-in-permanent-2}, and Proposition \ref{lem:quadratic-generator}. Since all the polynomials of type $(1)$ are minimal generators of $P_2(X)$, $G$ certainly generates the ideal. By definition, it suffices to show that $G$ is indeed a Gr{\"o}bner basis, i.e., that the $S$-polynomial $S(f,g)$ of any two polynomials  $f,g$ in $G$ reduces to zero with respect to $G$. Since the $S$-polynomial of two monomials is always zero, we can assume $f$ is of type $(1a)$ or $(1b)$.
	
	We note that the general case can be deduced once one shows that this is true when $X$ is a $6\times 6$ matrix. These finitely many cases can be either verified manually or by a computer. We will leave the complete verification to interested readers, and will only illustrate the case when $f$ and $g$ are both of type $(1a)$:
	\begin{align*}
		f&=x_{ii}x_{jj}+x_{ij}^2,\  i<j,\\
		g&=x_{kk}x_{ll}+x_{kl}^2,\  k<l.
	\end{align*}
	Without loss of generality, we can assume $i\leq k$. If $i,j,k,l$ are distinct, then $S(f,g)$ clearly reduces to zero. Hence we have three more cases:
	\begin{itemize}
		\item $k=i$: Without loss of generality we assume $j<l$. Then
		\begin{align*}
			S(f,g)&=(x_{ll})(x_{ii}x_{jj}+x_{ij}^2)-(x_{jj})(x_{ii}x_{ll}+x_{il}^2)\\
			&=(x_{ij})(x_{ij}x_{ll}+x_{il}x_{jl})-(x_{il})(x_{ij}x_{jl}+x_{il}x_{jj}).
		\end{align*}
		\item $k=j$: 
		\begin{align*}
			S(f,g)&=(x_{ll})(x_{ii}x_{jj}+x_{ij}^2)-(x_{ii})(x_{jj}x_{ll}+x_{jl}^2)\\
			&=(-x_{jl})(x_{ii}x_{jl}+x_{ij}x_{il})+(x_{ij})(x_{ij}x_{ll}+x_{il}x_{jl}).
		\end{align*}
		\item $l=j$: 
		\begin{align*}
			S(f,g)&=(x_{kk})(x_{ii}x_{jj}+x_{ij}^2)-(x_{ii})(x_{kk}x_{jj}+x_{kj}^2)\\
			&=(-x_{kj})(x_{ii}x_{kj}+x_{ij}x_{ik})+(x_{ij})(x_{ik}x_{kj}+x_{ij}x_{kk}).\qedhere
		\end{align*}
	\end{itemize}
\end{proof}

\section{Minimal primes of $P_2(X)$}\label{sec:4}

In this section we will determine all the minimal primes of $P_2(X)$, and compute the dimension and depth of the coordinate ring $\Bbbk[X]/P_2(X)$ along the way. For each pair of integers $i<j$, let $P_{ij}$ denote the ideal generated by
\[
\perm\begin{pmatrix}
	x_{ii}&x_{ij}\\
	x_{ij}&x_{jj}
\end{pmatrix}=x_{ii}x_{jj}+x_{ij}^2,
\]
and all the entries of $X$ outside of this submatrix.

\begin{theorem}\label{thm:symmetric-minimal-primes}
    Let $X$ be an $n\times n$ symmetric matrix with indeterminates as entries. An ideal $P$ is a minimal prime of $P_2(X)$ if and only if $P=P_{ij}$ for some $i<j$. In particular,
	\[
	\dim (\Bbbk[X]/ P_2(X)) = 2.
	\]
	Moreover, we have
	\[
	\depth (\Bbbk[X]/ P_2(X)) = \begin{cases}
		2 & \text{if } n=2,\\
		0 & \text{if } n>2
	\end{cases}.
	\]
\end{theorem}
\begin{proof}
    It is clear that $P_{ij}$ is prime for each pair of integers $i<j$. To show the converse, let $P$ be a minimal prime of $P_2(X)$. By Lemma \ref{lem:three-element-in-permanent-2}, for any distinct $i,j,k$, we have $x_{ii}^2x_{jj}x_{kk}\in P_2(X)\subseteq P$. Thus, there exist two different indices $i,j$ such that $x_{kk}\in P$ for any $k\neq i,j$. Moreover, for any $k\neq i,j$ and any $l\neq k$, we have
	\[
	x_{lk}^2+x_{ll}x_{kk}=\perm \begin{pmatrix}
		x_{ll}&x_{lk}\\
		x_{lk}&x_{kk}
	\end{pmatrix}\in P_2(X)\subseteq P.
	\]
	Thus $x_{lk}\in P$. Finally,  since
	\[
	\perm \begin{pmatrix}
		x_{ii}&x_{ij}\\
		x_{ij}&x_{jj}
	\end{pmatrix}\in P_2(X)\subseteq P,
	\]
	the ideal $P$ contains $P_{ij}$. Therefore, $\{P_{ij}\}_{i<j}$ is the set of minimal primes of $P_2(X)$.
	
	Next we compute $\depth(\Bbbk[X]/P_2(X))$. When $n=2$, the ring $\Bbbk[X]/P_2(X)$ is a hypersurface and hence its depth equals its Krull dimension, which is $2$. Now assume that $n>2$. Consider the monomial $f\coloneqq x_{12}x_{13}x_{23}$. We observe that $f=(x_{13})(x_{12}x_{23}+~x_{13}x_{22})-~(x_{13}^2x_{22})$, and $x_{13}^2x_{22}$ is not reducible by the Gr{\"o}bner basis $G$ in Theorem \ref{thm:Grobner-symmetric}. Thus $f\notin P_2(X)$. On the other hand, we note that the matrix $X$ is symmetric, and thus each variable in $f$ can be considered to be in one of two positions in $X$. Using Lemma~\ref{lem:three-element-in-permanent}, one can verify  that $f\mathfrak{m}\subseteq P_2(X)$, where $\mathfrak{m}$ is the homogeneous maximal ideal of $\Bbbk[X]$. Hence $\Hom_{\Bbbk[X]}(\Bbbk, \Bbbk[X]/P_2(X))\neq 0$, and thus $\depth(\Bbbk[X]/P_2(X))=0$ by \cite[Theorem 9.1]{24h}. 
\end{proof}

We present an application of this theorem.

\begin{corollary}\label{cor:symmetric-notCM-notreduced}
    Let $X$ be an $n\times n$ symmetric matrix with indeterminates as entries. If $n=2$, the ring $\Bbbk[X]/P_2(X)$ is a complete intersection domain. If $n>2$, the ring $\Bbbk[X]/P_2(X)$ is equidimensional, but not Cohen-Macaulay.
\end{corollary}

As another application, we will compute a set of generators for $\sqrt{P_2(X)}$. We recall Niermann's lemma:

\begin{lemma}[{\cite[p. 103]{Niermann}}]\label{lem:Niermann}
	Let $R$ be an arbitrary ring and $I_1,\dots, I_l,\ J_1,\dots, J_l$ be ideals of $R$ such that $I_i\subseteq J_j$ for any $i\neq j$. Then
	\[
	\bigcap_{i=1}^l (I_i+J_i)=(I_1+\cdots +I_l) +\bigcap_{i=1}^l J_i.
	\]
\end{lemma}

We can now compute $\sqrt{P_2(X)}$.

\begin{proposition}\label{prop:symmetric-generators-radical}
	Let $X$ be an $n\times n$ symmetric matrix with indeterminates as entries. We~have 
	\begin{multline*}
	    \sqrt{P_2(X)}=P_2(X)+\big(x_{ij}x_{kl}\mid i\neq j, k\neq l, \text{ and } (i,j)\neq (k,l)\big)+\big(x_{ij}x_{kk}\mid i,j,k \text{ are distinct}\big).
	\end{multline*}
	In particular, $P_2(X)$ is radical if and only if $n=2$.
\end{proposition}
\begin{proof}
	To simplify the notations, for any pair of integers $u<v$, let $I_{uv}$ denote the ideal generated by all the variables outside the submatrix $\begin{pmatrix}
		x_{uu}&x_{uv}\\x_{uv}&x_{vv}
	\end{pmatrix}$, and $J_{uv}$ the ideal generated by all $2\times 2$ subpermanents of $X$, excluding the permanent of this submatrix. Then $P_{uv}=J_{uv} + I_{uv}$. Since $J_{uv}\subseteq I_{ij}$ as long as $\{u,v\}\neq \{i,j\}$, we can use Niermann's lemma as follows:
	\begin{align*}
		\sqrt{P_2(X)}&=\bigcap_{1\leq u<v\leq n} P_{uv}\\
		&=\bigcap_{1\leq u<v\leq n}\left( J_{uv} + I_{uv}\right)\\
		&=\left( \sum_{1\leq u<v\leq n} J_{uv} \right)+ \bigcap_{1\leq u<v\leq n} I_{uv}\\
            &=P_2(X)+ \bigcap_{1\leq u<v\leq n} I_{uv}.
	\end{align*}
	We claim that 
	\begin{equation}\label{Niermann-thing}
		\bigcap_{1\leq u<v\leq n} I_{uv}=\!\begin{multlined}[t][.3\displaywidth]
			\big(x_{ij}x_{kl}\mid \text{two distinct off-diagonal entries of } X\big)\\+\big(x_{ij}x_{kk}\mid \text{one on- and one off-diagonal entries,} \\\text{  neither in the same row or column, of } X\big) \\+\big(x_{ii}x_{jj}x_{kk} \mid \text{three distinct on-diagonal entries of } X \big).
		\end{multlined}
	\end{equation}
	Since each $I_{uv}$ is a square-free monomial ideal, so is the intersection $\bigcap_{1\leq u<v\leq n} I_{uv}$. On the other hand, the containment $(\supseteq)$ of (\ref{Niermann-thing}) is straightforward. It suffices to show that if a square-free monomial $m$ of $\Bbbk[X]$ does not belong to the right-hand side of (\ref{Niermann-thing}), then $m$ does not belong to $\bigcap_{1\leq u<v\leq n} I_{uv}$, either. Indeed, we have the following cases:
	\begin{itemize}
		\item Suppose that $\deg(m)=1$, i.e., $m$ is just a variable. It is straightforward that there exists an ideal $I_{uv}$ such that $m\notin I_{uv}$.
		\item Suppose that $\deg(m)=2$, i.e., $m$ is the product of two different variables. Since $m$ does not belong to the  right-hand side of (\ref{Niermann-thing}), $m$ must be the product of two distinct on-diagonal entries of $X$, or the product of one on- and one off-diagonal entries of $X$ that are in the same row, or column. In other words, $m=x_{ii}x_{jj}$ or $m=x_{ij}x_{ii}$ for some integers $i\neq j$. Since the initial ideal of $I_{ij}$ with respect to the lex monomial ordering where  $x_{ij}$ is the largest  is $(x_{ij}^2)+(\text{variables other than }x_{ii},x_{ij}, x_{jj})$, which does not contain $m$. Therefore, in both cases, $m$ does not belong to $I_{ij}$
		\item Suppose that $\deg(m)=3$, i.e., $m$ is the product of three distinct variables. Since $m$ does not belong to the  right-hand side of (\ref{Niermann-thing}), we must have $m=x_{ii}x_{jj}x_{ij}$ for some integers $i<j$. Again, since  the initial ideal of $I_{ij}$ with respect to the lex monomial ordering where  $x_{ij}$ is the largest  is $(x_{ij}^2)+(\text{variables other than }x_{ii},x_{ij}, x_{jj})$, which does not contain $m$. Therefore, the ideal $I_{ij}$ itself does not contain $m$ either.
		\item Suppose that $\deg(m)\geq 4$, i.e., $m$ is divisible by at least four distinct variables. We will show that there is no such $m$, i.e., that $m$ always belongs to the right-hand side of (\ref{Niermann-thing}). Indeed, if $m$ is divisible by at least two different off-diagonal entries of $X$, then we are done. Otherwise $m$ is divisible by at most one off-diagonal entry of $X$, and in turn divisible by at least three distinct on-diagonal entries of $X$, which also implies what we claimed.
	\end{itemize}
	To sum up, the claim (\ref{Niermann-thing}) holds. Thus
	\begin{align*}
		\sqrt{P_2(X)}&=P_2(X)+\!\begin{multlined}[t][.3\displaywidth]
			 \big(x_{ij}x_{kl}\mid i\neq j, k\neq l, \text{ and } (i,j)\neq (k,l)\big)+\big(x_{ij}x_{kk}\mid i,j,k \text{ are distinct}\big)\\ +\big(x_{ii}x_{jj}x_{kk}\mid i,j,k \text{ are distinct}\big)
		\end{multlined} \\
	&=P_2(X)+\!\begin{multlined}[t][.3\displaywidth]
		\big(x_{ij}x_{kl}\mid i\neq j, k\neq l, \text{ and } (i,j)\neq (k,l)\big)+\big(x_{ij}x_{kk}\mid i,j,k \text{ are distinct}\big).
	\end{multlined} 
	\end{align*}
	The last equality is because for any distinct integers $u,v,w$, we have
	\[
	x_{uu}x_{vv}x_{ww}=(x_{uu})(x_{vv}x_{ww}+x_{vw}^2) -(x_{uu}x_{vw}^2)\in P_2(X)+\big(x_{ij}x_{kk}\mid i,j,k \text{ are distinct}\big).\qedhere
	\]
\end{proof}

We remark that the fact that $P_2(X)$ is radical exactly when $n=2$ can also be deduced from Theorem \ref{thm:symmetric-minimal-primes} since any $\mathbb{N}$-graded ring of depth $0$ is artinian, hence Cohen-Macaulay, which $P_2(X)$ is not if $n>2$. 

\section{A primary decomposition of $P_2(X)$}\label{sec:5}

We will compute a primary decomposition of $P_2(X)$ when $X$ is a symmetric $n\times n$ matrix. In the case of a generic matrix and a Hankel matrix, the ideal generated by its $2\times 2$ subpermanents only has one embedded prime, which is the maximal ideal. We shall see that in the case of a symmetric matrix, the ideal of its $2\times 2$ subpermanents have many more embedded~components. 

Set $R=\Bbbk[X]$. For each $1\leq k\leq n$, set
\begin{align*}
	Q_k&\coloneqq (x_{kk}x_{ij}+x_{ki}x_{kj}\mid i,j\neq k)+(x_{ij}\mid i,j\neq k)^2 + (x_{ij}\mid i,j\neq k)(x_{ki}\mid i\neq k),\\
	P_k&\coloneqq (x_{ij}\colon (i,j)\neq (k,k)).
\end{align*}

It is clear that for each $k$, we have $P_2(X)\subseteq Q_k$, the ideal $P_k$ is prime, and $\sqrt{Q_k}=P_k$. We will show that $P_k$ is $Q_k$-primary. The following proof is inspired by that of \cite[Proposition~3.2]{GGS07}.

\begin{proposition}\label{prop:Qk-primary}
	For each $1\leq k\leq n$, the ideal $Q_k$ is $P_k$-primary.
\end{proposition}

\begin{proof}
    By symmetry, we can assume $k=n$. Consider the lex monomial ordering where $x_{1n}>x_{2n}>\cdots >x_{nn}> x_{ij}$ for any $i,j\neq n$. Set 
    \[
    A=\{x_{in}x_{jn}+x_{nn}x_{ij},\ x_{ij} x_{i'j'},\ x_{ij}x_{i'n}\mid\  i,j,i',j'\in [n-1]\}.
    \]
    It is clear that $Q_n=(A)$. Set $t=x_{nn}$. We establish the following:
    \begin{enumerate}
        \item $A$ is a Gr\"obner basis;
        \item the leading coefficients of elements of $A$ do not involve $t$.
    \end{enumerate}
    Note that $t\notin P_n$. By (1) and (2), we have $\initial(Q_n)R_{P_n} \cap R = \initial(Q_n)R_t\cap R$, and hence $Q_nR_{P_n}\cap R=Q_nR_t\cap R$ by \cite[Proposition 3.6]{GTZ88}. We have
    \begin{align*}
        Q_nR_t\cap R &= (A)R_t \cap R  \tag{by (1)}\\
        &=(A,\ zt-1)R[z] \cap R  \tag{by the proof of \cite[Corollary 3.2 (v)]{GTZ88}}.
    \end{align*}
    Due to (1) and (2), $A\cup \{zt-1\}$ is a Gr\"obner basis as well. Thus by elimination theorem (see, e.g., \cite[Theorem 2.3.4]{Adams2012AnIT}), we have $Q_nR_t\cap R = (A)=Q_n$. In other words, $Q_n$ is its own $P_n$-primary component, and hence is primary, in particular, as desired. 
    
    It now suffices to prove (1) and (2). It is clear that (2) follows immediately from (1). To show (1), we will show that $S(f,g)$ reduces to $0$ with respect to $A$, for any $f,g\in A$. Indeed, it is known that $S(f,g)$ reduces to $0$ when $f$ and $g$ are both monomials (by definition), or when the leading terms of $f$ and $g$ are coprime (\cite[Lemma 3.3.1]{Adams2012AnIT}). We will verify the rest:
    \begin{align*}
        &S(x_{in}x_{jn}+x_{nn}x_{ij}, x_{in}x_{j'n}+x_{nn}x_{ij'}) = x_{nn}(x_{j'n}x_{ij}- x_{jn}x_{ij'}) \xrightarrow{x_{j'n}x_{ij}}  -x_{nn}x_{jn}x_{ij'} \xrightarrow{x_{jn}x_{ij'}} 0,\\
        &S(x_{in}x_{jn}+x_{nn}x_{ij}, x_{in}x_{i'j}) = x_{nn}x_{i'j}x_{ij} \xrightarrow{x_{i'j}x_{ij}} 0,
    \end{align*}
    for any $i,j,i',j'\in [n-1]$, as desired.
\end{proof}

When $n=2$, the ideal $P_2(X)$ is a hypersurface defined by an irreducible polynomial $x_{11}x_{22}+x_{12}^2$, and $P_2(X)$ is its own primary decomposition. For the rest of the section, we assume that $n>2$. We now show that $P_k$, for any $k\in [n]$, and the homogeneous maximal ideal $\mathfrak{m}=(x_{ij}\mid i,j\in [n])$, are associated primes of $P_2(X)$. First we recall a lemma.

\begin{lemma}\label{lem:primary-colon-equality}
    Let $Q$ be a primary ideal of $R$ and $x\notin \sqrt Q$ an element of $R$. Then $(Q\colon x)=Q$.
\end{lemma}

\begin{proof}
    It is clear that $Q\subseteq (Q\colon x)$. Conversely, consider $y\in (Q\colon x)$, i.e., $yx\in Q$. Since $x\notin \sqrt Q$ and $Q$ is $\sqrt Q$-primary, we have $y\in Q$. In other words,  $Q\supseteq (Q\colon x)$. The result then~follows.
\end{proof}

\begin{proposition}\label{prop:associated-primes-P2}
    Let $X$ be an $n\times n$ symmetric matrix with indeterminates as entries, where $n\geq 3$. Then
    \begin{enumerate}
        \item $P_2(X)\colon x_{12}^2x_{33} = \mathfrak{m}$;
        \item $P_2(X)\colon x_{ij}x_{kk}^2 = P_k$, for any distinct $i,j,k$.
    \end{enumerate}
    In particular, $\mathfrak{m}$ and $P_k$, where $k\in [n]$, are associated primes of $P_2(X)$.
\end{proposition}

\begin{proof}
    \begin{enumerate}
        \item $(\subseteq):$ Due to the Gr\"obner basis of $P_2(X)$ in Theorem~\ref{thm:Grobner-symmetric}, the monomial $x_{12}^2x_{33}$ is not in the initial ideal of $P_2(X)$ with respect to the diagonal monomial ordering. In particular, this means that  $x_{12}^2x_{33}\notin P_2(X)$. Thus  $P_2(X)\colon x_{12}^2x_{33} \neq R$, i.e., $P_2(X)\colon x_{12}^2x_{33}\subseteq \mathfrak{m}$.

        $(\supseteq):$ We want to show that $x_{12}^2x_{33}x_{ij}\in P_2(X)$ for any $i,j\in [n]$. Indeed, it is important to note that $x_{12}=x_{21}$. By Lemmas~\ref{lem:three-element-in-permanent} and \ref{lem:three-element-in-permanent-2}, the result would follow if $x_{ij}, x_{12}, x_{33}$ are three distinct entries in at least two different rows (resp, columns) and three different columns (resp, rows). Therefore, the only remaining cases are when $x_{ij}, x_{12}, x_{33}$ are two distinct entries, i.e., $x_{ij}$ equals $x_{12}$ or $x_{33}$. In this case, the result follows from Lemma~\ref{lem:three-element-symmetric}, as desired.
        
        \item By symmetry, we can assume that $i=1, j=2,$ and $k=n$. We thus want to show that $P_2(X)\colon x_{12}x_{nn}^2 = P_n$.
        
        $(\supseteq):$ The arguments are similar to those above. We want to show that $x_{12}x_{nn}^2x_{ij}$ for any $i,j\in [n]$ such that $(i,j)\notin (n,n)$. The result would follow from Lemmas~\ref{lem:three-element-in-permanent} and \ref{lem:three-element-in-permanent-2} if $x_{12}, x_{nn}, x_{ij}$ are three distinct entries. The remaining case is when $x_{ij}=x_{12}$, in which case the result follows from Lemma~\ref{lem:three-element-symmetric}.

        $(\subseteq):$ We have 
        \[
        P_2(X)\colon x_{12}x_{nn}^2\subseteq Q_n\colon x_{12}x_{nn}^2 = (Q_n\colon x_{nn}^2) \colon x_{12} = Q_n\colon x_{12},
        \]
        where the last equality follows from Lemma~\ref{lem:primary-colon-equality}. It now suffices to show that $Q_n\colon x_{12}\subseteq P_n$. We will show this by contraposition. Consider a homogeneous polynomial $f\notin P_n$. Modulo $P_n$, we can assume that $f=x_{nn}^r$ for some $r>0$. We consider the lex monomial ordering where $x_{1n}>x_{2n}>\cdots >x_{nn}> x_{ij}$ for any $i,j\neq n$. From the proof of Proposition~\ref{prop:Qk-primary}, the set of generators of $Q_n$
        \[
        \{x_{in}x_{jn}+x_{nn}x_{ij},\ x_{ij} x_{i'j'},\ x_{ij}x_{i'n}\mid\  i,j,i',j'\in [n-1]\}
        \]
        is a Gr\"obner basis with respect to this ordering. We then have $\initial(fx_{12})=x_{12}x_{nn}^r\notin \initial(Q)$. In particular, this implies that $fx_{12}\notin Q_n$, or equivalently, $f\notin Q\colon x_{12}$, as~desired.\qedhere
    \end{enumerate}
\end{proof}

Couple these associated primes with the minimal primes we found in Theorem~\ref{thm:symmetric-minimal-primes}, we have obtained all associated primes of $P_2(X)$. We will prove this by showing a primary decomposition of $P_2(X)$. We recall the following lemma.

\begin{lemma}[{\cite[Fact 4.2]{GGS07}}]\label{lem:primary-procedure}
    Let $I$ be a homogeneous ideal of $R$, and $x\notin \sqrt{I}$ an element of $R$. Then there exists an integer $n$ such that
    \[
    (I\colon x^n) =(I\colon x^{n+1}).
    \]
    For such an $n$, we have
    \[
    I=(I\colon x^n)\cap (I+(x^n)).
    \]
\end{lemma}

This result provides a pathway towards finding the primary decomposition of an ideal $I$, by instead determining the primary decomposition of the bigger ideals $I\colon x^n$ and $I+(x^n)$. It is worth noting that this procedure does not imply that any ideal in the process is primary. In our case, Lemma~\ref{prop:Qk-primary} will become essential for this reason. Lemma~\ref{lem:primary-procedure} also comes with a strong hypothesis regarding an equality between colon ideals. The next result is how we will obtain this~hypothesis.

\begin{lemma}\label{lem:colon-equality}
    Let $I$ and $Q$ be ideals of $R$, $x$ an element in $R$, and $n$ a positive integer. Assume that the following holds:
    \begin{enumerate}
        \item $Q$ is $P$-primary where $P=\sqrt{Q}$ is a prime ideal;
        \item $x\notin P$;
        \item $I\subseteq Q$;
        \item $Q\subseteq (I\colon x^n)$.
    \end{enumerate}
    Then we have $Q=(I\colon x^n)=(I\colon x^{n+1})$.
\end{lemma}

\begin{proof}
    First of all we observe that no power of $x$ is in $P$ as $P$ is prime and $x\notin P$. By Lemma~\ref{lem:primary-colon-equality}, we have $Q=(Q\colon x^{n+1})$. Therefore, we have
    \[
    Q\subseteq (I\colon x^n) \subseteq (I\colon x^{n+1}) \subseteq (Q\colon x^{n+1})=Q.
    \]
    Thus the result follows.
\end{proof}

Many colon ideals will appear when we apply Lemma~\ref{lem:primary-procedure}. For this reason, we will provide some equalities below.

Set $\Omega_{\off}=\{x_{ij}\mid i,j\in [n], i\neq j\}$ and $\Omega_{\on}=\{x_{ii}\mid i\in [n]\}$. In other words, $\Omega_{\off}$ is the set of off-diagonal entries of $X$, and $\Omega_{\on}$ is the set of diagonal entries. For a set $\sigma$ of variables ad an integer $n$, let $\sigma^{[n]}$ denote the set $\{x^n\mid x\in \sigma\}$.

\begin{proposition}\label{prop:primary-equality}
    For each $\sigma\subset \Omega_{\off}$ and $x_{ij}\in \Omega_{\off}\setminus \sigma$ where $i\neq j$, we have
    \[
    P_{ij}=(P_2(X) + (\sigma^{[3]})) : x_{ij}^3 = (P_2(X) + (\sigma^{[3]})) : x_{ij}^4.
    \]
\end{proposition}

\begin{proof}
    It suffices to show the four conditions in Lemma~\ref{lem:colon-equality} as the result would follow immediately. Indeed, $P_{ij}$ is prime, and thus $P_{ij}$-primary. It is clear that $x_{ij}\notin P$ and $P_2(X)+(\sigma^{[3]})\subseteq P_{ij}$ by definition. It now suffices to show that $P_{ij}\subseteq (P_2(X) + (\sigma^{[3]})) : x_{ij}^3$. Recall~that
    \[
    P_{ij}=(x_{ii}x_{jj}+x_{ij}^2,\ x_{ik}, \ x_{jk}, \ x_{kl} \mid k,l\notin \{i,j\}).
    \]
    It is clear that $x_{ii}x_{jj}+x_{ij}^2\in P_2(X)$, and for each $k,l\notin \{i,j\}$, the variables $x_{ik}$ and $ x_{jk}$ are in $P_2(X) : x_{ij}^3$ by Lemma~\ref{lem:three-element-symmetric}. Moreover, $x_{kl}$ is in $P_2(X)\colon x_{ij}^3$ when $l=k$ by Lemma~\ref{lem:three-element-symmetric}, and also when $l\neq k$ by Lemma~\ref{lem:quadratic-generator}. In other words, we have $P_{ij}\subseteq P_2(X)  : x_{ij}^3\subseteq (P_2(X) + (\sigma^{[3]})) : x_{ij}^3$, as desired.
\end{proof}

\begin{proposition}\label{prop:primary-equality-2}
    For each $\tau \subset \Omega_{\on}$ and $x_{kk}\in \Omega_{\on}\setminus \tau$, we have
    \[
    Q_k=(P_2(X) + (\Omega_{\off}^{[3]}) + (\tau^{[2]})) : x_{kk}^2 = (P_2(X) + (\Omega_{\off}^{[3]}) + (\tau^{[2]})) : x_{kk}^3.
    \]
\end{proposition}

\begin{proof}
    It suffices to show the four conditions in Lemma~\ref{lem:colon-equality} as the result would follow immediately. Indeed, $Q_k$ is $P_{k}$-primary by Proposition~\ref{prop:Qk-primary}. It is clear that $x_{kk}\notin P_k$. Next we recall that
    \[
    Q_k= (x_{kk}x_{ij}+x_{ki}x_{kj}\mid i,j\neq k)+(x_{ij}\mid i,j\neq k)^2 + (x_{ij}\mid i,j\neq k)(x_{ki}\mid i\neq k).
    \] 
    Consider $x_{ij}^3\in \Omega_{\off}^{[3]}$ for some $i\neq j$. If $i,j\neq k$, then $x_{ij}^3\in Q_k$ by definition. On the other hand, if $i=k$, then $x_{ij}^3=x_{kj}^3=(x_{kk}x_{jj}+x_{kj}^2)(x_{kj}) - (x_{jj}x_{kj})(x_{kk})\in Q_k$. In other words, we have $(\Omega_{\off}^{[3]})\subseteq Q_k$ either way, and thus 
    $P_2(X) + (\Omega_{\off}^{[3]}) + (\tau^{[2]})\subseteq Q_k$. 
    
    It now suffices to show that $Q_{k}\subseteq (P_2(X) + (\Omega_{\off}^{[3]}) + (\tau^{[2]})) : x_{kk}^2$. Consider $i,j\neq k$. We have $(x_{kk}x_{ij}+x_{ki}x_{kj}^2)\in P_2(X) \subseteq (P_2(X) + (\Omega_{\off}^{[3]}) + (\tau^{[2]})) : x_{kk}^2$. Next we show that the monomial $x_{ij}x_{i'j'}$ is in $(P_2(X) + (\Omega_{\off}^{[3]}) + (\tau^{[2]})) : x_{kk}^2$ for any $i',j'\neq k$. Indeed, we have
    \[
    x_{ij}x_{i'j'}x_{kk}^2 = \begin{cases}
        x_{ij}x_{i'j'}x_{kk}^2 &\text{if $i\neq i'$ and $j\neq j'$},\\
        (x_{ij}x_{ij'}x_{kk})(x_{kk}) &\text{if $i= i'$ and $j\neq j'$},\\
        (x_{ij}x_{i'j}x_{kk})(x_{kk}) &\text{if $i\neq i'$ and $j= j'$},\\
        x_{ij}^2x_{kk}^2 &\text{if $i= i'$, $j= j'$, and $i\neq j$},\\
        x_{ii}^2x_{kk}^2 &\text{if $i= i'$, $j= j'$, and $i= j$}.\\
    \end{cases} 
    \]
    We then have $x_{ij}x_{i'j'}x_{kk}^2\in P_2(X)$ in the first case by Lemma~\ref{lem:three-element-in-permanent-2}, in the second and third cases by Lemma~\ref{lem:three-element-in-permanent}, and in the fourth case by Lemma~\ref{lem:three-element-symmetric}. In the fifth case, modulo $P_2(X)$, we have $ x_{ij}x_{i'j'}x_{kk}^2= (x_{ii}x_{kk})^2=(-x_{ik}^2)^2=x_{ik}^4\in \Omega_{\off}^{[3]}$. Therefore, $x_{ij}x_{i'j'}$ is in $(P_2(X) + (\Omega_{\off}^{[3]}) + (\tau^{[2]})) : x_{kk}^2$, as claimed.

    Finally, we show that $x_{ij}x_{kl}$ is in $(P_2(X) + (\Omega_{\off}^{[3]}) + (\tau^{[2]})) : x_{kk}^2$ for any $l\neq k$. Indeed, we must have $l\neq i$ or $l\neq j$, and thus without loss of generality, we assume the latter holds. Then $x_{ij}, x_{kl},$ and $x_{kk}$ are in two different rows and three different columns, and thus $x_{ij}x_{kl}$ is in $(P_2(X) + (\Omega_{\off}^{[3]}) + (\tau^{[2]})) : x_{kk}^2$ by Lemma~\ref{lem:three-element-in-permanent}, as claimed.
    
    To summarize, we have $Q_{k}\subseteq (P_2(X) + (\Omega_{\off}^{[3]}) + (\tau^{[2]})) : x_{kk}^2$, as desired.
\end{proof}

\begin{theorem}\label{thm:primary-decomposition}
    Let $X$ be an $n\times n$ symmetric matrix with indeterminates as entries, where $n\geq 3$. The intersection
    \[
    \left( \bigcap_{1\leq i<j\leq n} P_{ij} \right) \cap \left( \bigcap_{k=1}^n Q_k \right) \cap \left(P_2(X)+(x_{ij}^3\mid 1\leq i<j\leq n)+ (x_{kk}^2\mid k\in [n])\right)
    \]
    is a irredundant primary decomposition of $P_2(X)$.
\end{theorem}
\begin{proof}
    The ideal $P_{ij}$ is prime, $Q_{k}$ is primary (Lemma~\ref{prop:Qk-primary}), and $P_2(X)+(x_{ij}^2\mid i,j\in [n])$ is primary since its radical is the homogeneous maximal ideal. Thus the above intersection is indeed a primary decomposition. Next we show that this intersection is indeed $P_2(X)$. By applying Lemma~\ref{lem:primary-procedure} and Proposition~\ref{prop:primary-equality} repeatedly, we obtain:
    \begin{align*}
        P_2(X) &= \left( P_2(X) \colon x_{12}^3 \right) \cap \left( P_2(X)+(x_{12}^3) \right)\\
        &= P_{12} \cap \left( P_2(X)+(x_{12}^3) \right)\\
        &=P_{12}\cap \left( \big( P_2(X)+(x_{12}^3)\big) \colon x_{13}^3 \right) \cap \left(  P_2(X)+(x_{12}^3)+(x_{13}^3) \right)\\
        &=P_{12}\cap P_{13} \cap \left(  P_2(X)+(x_{12}^3,x_{13}^3) \right)\\
        &\cdots\\
        &= \left( \bigcap_{1\leq i<j\leq n} P_{ij} \right) \cap \left(P_2(X)+(\Omega_{\off}^{[3]})\right).
    \end{align*}
    Next we apply  Lemma~\ref{lem:primary-procedure} and Proposition~\ref{prop:primary-equality-2} repeatedly:
    \begin{align*}
        P_2(X) &=\left( \bigcap_{1\leq i<j\leq n} P_{ij} \right) \cap \left(P_2(X)+(\Omega_{\off}^{[3]})\right)\\
        &=\left( \bigcap_{1\leq i<j\leq n} P_{ij} \right) \cap \left( \left(P_2(X)+(\Omega_{\off}^{[3]}) \right) \colon x_{11}^2 \right) \cap  \left(P_2(X)+(\Omega_{\off}^{[3]})  + (x_{11}^2) \right)\\
        &=\left( \bigcap_{1\leq i<j\leq n} P_{ij} \right) \cap Q_1 \cap  \left(P_2(X)+(\Omega_{\off}^{[3]})  + (x_{11}^2) \right)\\
        &\cdots\\
        &=\left( \bigcap_{1\leq i<j\leq n} P_{ij} \right) \cap \left( \bigcap_{k=1}^n Q_k \right) \cap \left(P_2(X)+(x_{ij}^3\mid 1\leq i<j\leq n)+ (x_{kk}^2\mid k\in [n])\right),
    \end{align*}
    as claimed. Finally, we note that the radicals of the components in this decomposition are distinct prime ideals, and these are associated primes of $P_2(X)$ by Theorem~\ref{thm:symmetric-minimal-primes} and Proposition~\ref{prop:associated-primes-P2}. Therefore, the decomposition is irredundant, as desired.
\end{proof}

\section*{Acknowledgement} This paper is an extension on the author's Ph.D thesis \cite{TC}. The author was supported by the NSF grants DMS 1801285, 2101671, and 2001368, and the Infosys Foundation. The author would like to thank his advisor Anurag K. Singh for suggesting this problem and for the constant encouragement and helpful discussions.  The author would like to thank Vaibhav Pandey, Irena Swanson, and Uli Walther for many helpful suggestions. Part of this work was done while the author visited Purdue University. The author would like to thank the Department of Mathematics at Purdue, and especially Annie and Bill Giokas, for their hospitality.
\bibliographystyle{amsplain}
\bibliography{refs}
\end{document}